\documentclass[a4paper]{gtart}

\usepackage[inner=20mm, outer=20mm, textheight=245mm]{geometry}

\usepackage{psfrag, graphicx, epsfig,color}

\usepackage{amsmath, amssymb, latexsym, euscript, amsthm}

\usepackage{thmtools}

\usepackage[colorlinks,citecolor=blue,linkcolor=blue, backref=page]{hyperref}
\usepackage[nameinlink]{cleveref}

\usepackage[margin=1cm]{caption}

\usepackage{tikz}
\usetikzlibrary{matrix}

\usepackage{mathptmx, wrapfig}

\newcommand{\cT}{\mathcal{T}}

\usepackage{algorithmicx}
\usepackage{algpseudocode}

\usepackage{longtable}

\usepackage[normalem]{ulem}

\usepackage[colorinlistoftodos]{todonotes}
\theoremstyle{plain}
\newtheorem{theorem}{Theorem}%[section]
\newtheorem*{theorem*}{Theorem}

\newtheorem{proposition}[theorem]{Proposition}
\newtheorem{corollary}[theorem]{Corollary}
\newtheorem{scholion}[theorem]{Scholion}

\theoremstyle{definition}

\newtheorem*{definition*}{Definition}

\newtheorem{question}[theorem]{Question}

\theoremstyle{remark}

\numberwithin{equation}{section}

\begin{document}

\title{Moving between 3-manifold triangulations is NP-hard}
\author{Stephan Tillmann and Anastasiia Tsvietkova}

\begin{abstract} 
We show that \textsc{number of bistellar moves and sparse degree-two edge collapses for 3-sphere} is NP-hard. It follows that a similar problem for an arbitrary 3-manifold is NP-hard as well. This is the first NP-hardness result concerning moves between two triangulations of a 3-manifold.
\end{abstract}

%MSC2020
%57K31 Invariants of 3-manifolds (also skein modules; character varieties)
%57M15 - Relations of low-dimensional topology with graph theory
%57M50 - General geometric structures on low-dimensional manifolds
%57Q15 - Triangulating manifolds
%68Q25 - Analysis of algorithms and problem complexity
%68W40 - Analysis of algorithms

\makeshorttitle

%%%%%%%%%%%%%%%%%%%%%%%%%%%
%%%%%%%%%%%%%%%%%%%%%%%%%%%

\section{Introduction}
\label{sec:intro}

\subsection{Main result.}
One of the key facts in 3-dimensional topology is that any two triangulations of a 3-manifold can be connected by a finite sequence of certain moves. Out of them, the finite set of bistellar moves, introduced by Pachner \cite{Pachner-bistellare-1978} plays a central role and has been used to study invariants of and structures on 3-manifolds. In modern algorithms, crushing and edge collapses are often additionally used, since they can be more efficient in practice. For instance, they are implemented in the computational topology software {\tt Regina}\rm~\cite{regina}. The set of edge collapses is infinite. We restrict our attention to a particular type of edge collapse that we call a \emph{sparse degree-two edge collapse} and consider the following decision problem. 

\noindent \textsc{number of bistellar moves and sparse degree-two edge collapses for closed 3-manifold.} Given a natural number $k$ and two topological triangulations $\cT_1, \cT_2$ of a 3-manifold, can one get from $\cT_1$ to $\cT_2$ using at most $k$ bistellar moves, sparse degree-two edge collapses, and their inverses?

We define a new class of triangulations of 3--sphere associated with embeddings of planar graphs (\Cref{subsec:construction}). The associated graph has a Hamiltonian path if and only if the triangulation can be simplified with a given number of certain type of edge collapses. We use this fact to perform Karp reduction from a verison of \textsc{hamiltonian path} problem, known to be NP-hard, to a modification of the above decision problem (\Cref{claim:collapses_iff_hamiltonian} and \Cref{thm:sparse_is_NP-complete}). We then show that if the set of moves is extended to also include bistellar moves, the minimal sequence of moves stays unchanged. This gives the main result:

\begin{theorem} 
\label{thm:main}
\textsc{number of bistellar moves and sparse degree-two edge collapses for 3-sphere} is NP-hard.
\end{theorem}

This immediately implies the following: 

\begin{corollary} 
\label{cor:main}
\textsc{number of bistellar moves and sparse degree-two edge collapses for closed 3-manifold} is NP-hard.
\end{corollary}

\subsection{Relation to \textsc{homeomorphism of 3-manifolds}.}

The following problem is at the heart of the topological study of 3-manifolds.

\textsc{3-manifold homeomorphism.} Given two triangulations of 3-manifolds, are the manifolds homeomorphic?

It is known to be decidable \cite{ScottShort}, with various theoretical results involved in the solution, including the proof of Thurston's Geometrisation Conjecture. The decidability implies that an upper bound on the number of moves required to pass from one triangulation to another of the same 3--manifold is computable as a function on input the sizes of the two triangulations (Theorem 4.3 in \cite[\S 4]{Lackenby-algorithms-2022}).  

The current upper bound on the complexity of \textsc{3-manifold homeomorphism} is $2^{2^{...^{2^t}}}$ due to Kuperberg \cite{Kuperberg}, where $t$ is the sum of the number of tetrahedra in the given triangulations, and the height of the tower is a universal but unknown constant. The only known lower bound is due to Lackenby \cite{LackenbyNPhard}: \textsc{3-manifold homeomorphism} is at least as hard as \textsc{graph isomorphism}, which has recently been shown to be solvable in quasi-polynomial time by Babai \cite{Babai}. A more specific lower bound is an open question, recorded, for example, in the list of problems in low-dimensional topology known as \emph{Kirby's list} \cite{K3}. Some restricted versions of this problem lie in NP, e.g. recognition of the 3--sphere~\cite{Ivanov, Schleimer}, Seifert fibered spaces with boundary~\cite{Jackson-recognition-2025}, and elliptic 3--manifolds~\cite{Lackenby-recognising-2026}. A survey by Lackenby~\cite[\S 4]{Lackenby-algorithms-2022} discusses this further.

A related decision problem is that of \textsc{link isotopy}. An algorithm for \textsc{link isotopy} can be obtained from an algorithm for \textsc{3-manifold homeomorphism} as shown by Matveev~\cite[Corollary 6.1.4]{matveev_book}. The existence of a polynomial algorithm for \textsc{link isotopy} in general is an open problem~\cite{K3}, and is known for the class of links given by alternating diagrams, as recently shown by Haider and Tsvietkova~\cite{HT}. For related decision problems, see another survey by Lackenby~\cite{LackenbyKnotTheory}.  

\subsection{Lower bounds on algorithmic complexity of problems on 3-manifolds.}

New strategies to prove NP-hardness can be of independent interest, especially given that discrete complexity theory still has many open questions. Our Karp reduction is based on a new construction, different from existing NP-hardness proofs in 3-dimensional topology. A number of such NP-hardness results exist. A subset of them concerns properties of links and knots and their projections on the 2--sphere: the sublink problem \cite{LackenbyNPhard}, $k$-component unlink as a sublink problem \cite{Koenig-NP-2021, deMesmayRieckSedgwickTancer, Rieck}, unlinking and unknotting in $k$ Reidemeister moves \cite{Koenig-NP-2021, deMesmayRieckSedgwickTancer}, splitting and unlinking numbers of links \cite{Koenig-NP-2021, KoenigTsvietkova1}, diagrammatic unlinking number \cite{Koenig-NP-2021}, crossing number of a link \cite{deMesmaySchaeferSedgwick}, alternating $k$-component sublink \cite{Koenig-NP-2021}. Another subset concerns embedded surfaces in link and knot complements: knot genus in an arbitrary manifold \cite{AgolHassThurston}, Thurston complexity of a link in 3-sphere \cite{LackenbyNPhard}, embedded non-orientable surface of given Euler characteristic in a 3-manifold \cite{BurtondeMesmayWagner}. Yet another set concerns embedding of complexes of various dimension into manifolds of various dimension, such as curves in manifolds of dimension 3~\cite{deMesmayRieckTancer}. The problems of detecting taut angle structures on a triangulation \cite{BurtonSpreer} and computating optimal Morse matchings~\cite{Joswig-computing-2006} are perhaps closest in spirit to the topic of this paper, though these NP-hardness results concerns a property of a fixed 3-manifold triangulation rather than a property of a 3-manifold and its set of triangulations.

One can modify the set of moves in the decision problem we consider. One direction is to investigate it for other types of crushing and edge collapse moves appearing in the literature. Another direction would be restricting the set of moves. For example, the following question is still open.

\begin{question}

Is \textsc{number of bistellar moves for a 3-manifold} NP--hard? 
\end{question}

\paragraph{Acknowledgements.} 
This research was partially funded by the International Visitor Scheme of the Sydney Mathematical Research Institute (SMRI). Research of Tillmann was supported in part under the Australian Research Council's Discovery funding scheme (project number DP190102259). Tsvietkova was partially supported by NSF DMS-2005496 and NSF CAREER DMS-2142487 grants, and Rutgers Board of Trustees Research Fellowship for Scholarly Excellence.

%%%%%%%%%%%%%%%%%%%%%%%%%%%
%%%%%%%%%%%%%%%%%%%%%%%%%%%

\section{Moves between 3--manifold triangulations}
\label{sec:moves}

The triangulations considered in this paper are defined in \Cref{subsec:triangulations}. Given triangulations $\cT$ and $\cT'$, we write $\cT \mapsto \cT'$ if the triangulations are related by a single \textbf{elementary move} that is one of the \textbf{bistellar moves} defined in \Cref{subsec:bistellar}, or an \textbf{sparse degree-two edge collapse} or its inverse as defined in \Cref{subsec:edge_collapse}.

%%%%%%%%%%%%%%%%%%%%%%%%%%%

\subsection{Triangulations}
\label{subsec:triangulations}

We refer the reader to \cite{Rubinstein-traversing-2019} for a detailed literature review of different types of 3--manifold triangulations and results that describe how to connect any two such triangulations by local moves. In this paper, we are only concerned with the following triangulations of closed 3--manifolds. 

Let $\widetilde{\Delta}$ be a finite union of pairwise disjoint euclidean $3$--simplices. 
Every $k$--simplex $\tau$ in $\widetilde{\Delta}$ is contained in a unique $3$--simplex $\sigma_\tau.$ A $2$--simplex in $\widetilde{\Delta}$ is called a \emph{facet}.
Let $\Phi$ be a family of affine isomorphisms pairing the facets in $\widetilde{\Delta},$ with the properties that $\varphi \in \Phi$ if and only if $\varphi^{-1}\in \Phi,$ and every facet is the domain of a unique element of $\Phi.$ The elements of $\Phi$ are termed \emph{face pairings}.

The quotient space $P = \widetilde{\Delta}/\Phi$ with the quotient topology is then a \emph{closed $3$--dimensional pseudo-manifold}, and the quotient map is denoted $p\co \widetilde{\Delta} \to P.$ 
We will always assume that $P$ is connected. As shown in \cite[\S2.2]{Rubinstein-traversing-2019}, the quotient space $P$ is a manifold if and only if
\begin{enumerate}
\item the restriction of the map $p\co \widetilde{\Delta} \to P$ to the interior of each {$1$-simplex} in $\widetilde{\Delta}$ is a homeomorphism, and
\item the link of each vertex in $P$ is a 2--sphere.
\end{enumerate}

Given a closed 3--manifold $M$, we say that the triple $\cT = ( \widetilde{\Delta}, \Phi, h)$ is a \emph{triangulation} of $M$ if $h\co P \to M$ is a piecewise linear homeomorphism.

The images of vertices, edges, triangles and tetrahedra under the
composition of the map $p\co \widetilde{\Delta} \to P$ with $h$ give us the 
vertices, edges, triangles and tetrahedra of the triangulation in $M.$ These generally have self-identifications along their boundaries. The \textbf{degree} of an edge $e \subset M$ is the number of 1--simplices in the preimage $p^{-1}(e).$ 

A key result of Moise~\cite{Moise-affine-1952} states that every three-dimensional manifold can be triangulated in the way described above.

For the triangulation $\cT$ of a closed 3--manifold $M,$ let 
\begin{itemize}
\item $t$ be the number of tetrahedra in $\cT$, and
\item $v$ be the number of vertices in the triangulation. 
\end{itemize}
We then define the \textbf{complexity pair} of $\cT$ to be 
\[
c(\cT) = (t,\; v).
\] 

%%%%%%%%%%%%%%%%%%%%%%%%%%%
\subsection{Bistellar moves}
\label{subsec:bistellar}

A finite set of local moves that change one triangulation to another are the \emph{bistellar moves} that have been popularised by Pachner~\cite{Pachner-bistellare-1978}. Previous work by Alexander~\cite{Alexander-combinatorial-1930} and Newman~\cite{Newman-foundation-1926} used an infinite set of stellar moves.
See~\cite{Rubinstein-traversing-2019} for a detailed discussion of these results and more recent work.

The bistellar moves for 3--manifold triangulations are as follows. The 1-4 move introduces a vertex inside a tetrahedron and connects it to the four vertices of the tetrahedron with four edges. This creates six new triangular faces, spanned by the edges of the tetrahedron and the new vertex, and four new tetrahedra. The inverse move is called a 4-1 move. Then there is the 2-3 move, which is performed on two different tetrahedra meeting in a triangular face. The 2-3 move deletes this face by introducing a new edge connecting the opposite corners of the tetrahedra. Its inverse is the 3-2 move. Each of these moves changes the number of tetrahedra, faces and edges. However, the 2-3 and 3-2 moves do not affect the number of vertices in the triangulation. 

In summary, we have the following effects by these elementary moves on the complexity pair:
\begin{align*}
B_{1-4}:(t,\; v)  &\mapsto (t,\; v) +(\phantom{-}3,\; \phantom{-}1)\\
B_{2-3}:(t,\; v)  &\mapsto (t,\; v) +(\phantom{-}1,\; \phantom{-}0)\\
B_{3-2}:(t,\; v)  &\mapsto (t,\; v) +(-1,\; \phantom{-}0)\\
B_{4-1}:(t,\; v)  &\mapsto (t,\; v) +(-3,\; -1)
\end{align*}

\begin{figure}[htbp]
\centering
\includegraphics[width=0.8\textwidth]{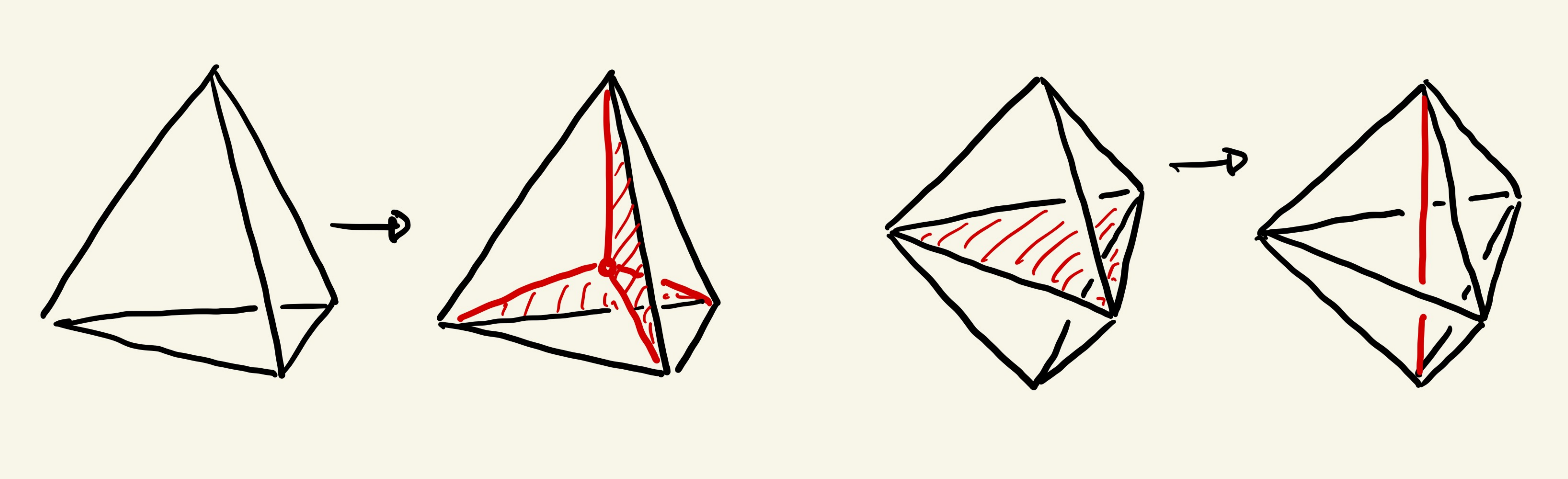}
\caption{The 1-4 and 2-3 moves.} 
\label{1-4_and_2-3_moves}
\end{figure}

The following result shows that the moves defined above play the same role for closed 3-manifolds as the Reidemeister moves in knot theory:
\begin{theorem}[Alexander, Newman, Moise, Pachner]
\label{pachner}
The set of all triangulations of a closed three-dimensional manifold $M$ is connected under 1-4, 2-3, 3-2 and 4-1 moves.
\end{theorem}

%%%%%%%%%%%%%%%%%%%%%%%%%%%

\subsection{Sparse degree-two edge collapse}
\label{subsec:edge_collapse}

Edge collapses were introduced by Burton~\cite[\S 2.2.4]{Burton-computational-2013}. In this paper, we only consider edge collapses of certain degree-two edges that are more restrictive than those allowed in \cite{Burton-computational-2013}. Suppose $e$ is an edge of degree two in $M$ and denote $\Delta_1$ and $\Delta_2$ the tetrahedra containing $e.$ We require that $\Delta_1 \cup \Delta_2$ is 3--ball embedded in the interior of $M.$ In particular, for each $i \in \{ 1, 2\},$
\begin{itemize}
\item $\Delta_i$ has four pairwise distinct vertices; 
\item  there is an edge $f_i \subset \Delta_i$ that is not shared with the other tetrahedron. 
\end{itemize}
If the edges $f_i$ both have degree at least three, then we call $e$ a \textbf{sparse degree-two edge}. \Cref{fig:deg2collapse} shows $\Delta_1 \cup \Delta_2,$ where tetrahedra $\Delta_1$ and $\Delta_2$ share the edge $e$ of degree two between distinct vertices $R_1$ and $R_2,$ and are spanned by the additional vertices $N$ and $S.$ 

The \textbf{sparse degree-two edge collapse} of $e$ is the following homotopy of $M$ that takes the 3--ball $\Delta_1 \cup \Delta_2$ to a 2--ball that is the union of 2--faces, and such that it takes the triangulation of $M$ to a triangulation with two tetrahedra fewer. The homotopy identifies $R_1$ and $R_2$ and flattens each tetrahedron to a face with three pairwise distinct vertices. By hypothesis, the edges $f_1$ and $f_2$ between the two other vertices $N$ and $S$ are distinct before and after the homotopy and have degree at least two in the new triangulation. Note that the edges $g_i$ and $h_i$ have degree at least three since $f_1 \neq f_2,$ and hence each of $\Delta_1$ and $\Delta_2$ is incident with a unique edge of degree two.\footnote{The adjective sparse is meant to indicate that even if a triangulation admits a collapse of a degree-two edge as defined in \cite{Burton-computational-2013}, then it may not be that of a sparse degree-two edge.}

By construction, the sparse degree-two edge collapse changes the triangulation, but not the underlying identification space. Suppose that for the previously defined quotient map $p$ and homeomorphism $h$, we have $h\circ p(\widetilde{\Delta}_i) = \Delta_i$, where $\widetilde{\Delta}_i \in \widetilde{\Delta}.$ Then the edge collapse gives us a new triangulation of $M$ with set of tetrahedra $\widetilde{\Delta} \setminus \{ \widetilde{\Delta}_1, \widetilde{\Delta}_2\}.$ The new face pairings are obtained by deleting all face pairings from $\Phi$ with domain or target in each $\Delta_i$, $i\in \{1, 2\}$, and introducing new face pairings between the four free faces as determined by the homotopy. In particular, on the level of the triangulation, the sparse degree-two edge collapse and the resulting triangulation are completely encoded by specifying the two tetrahedra sharing the sparse degree-two edge. 

\begin{figure}[htbp]
\centering
\includegraphics[width=0.8\textwidth]{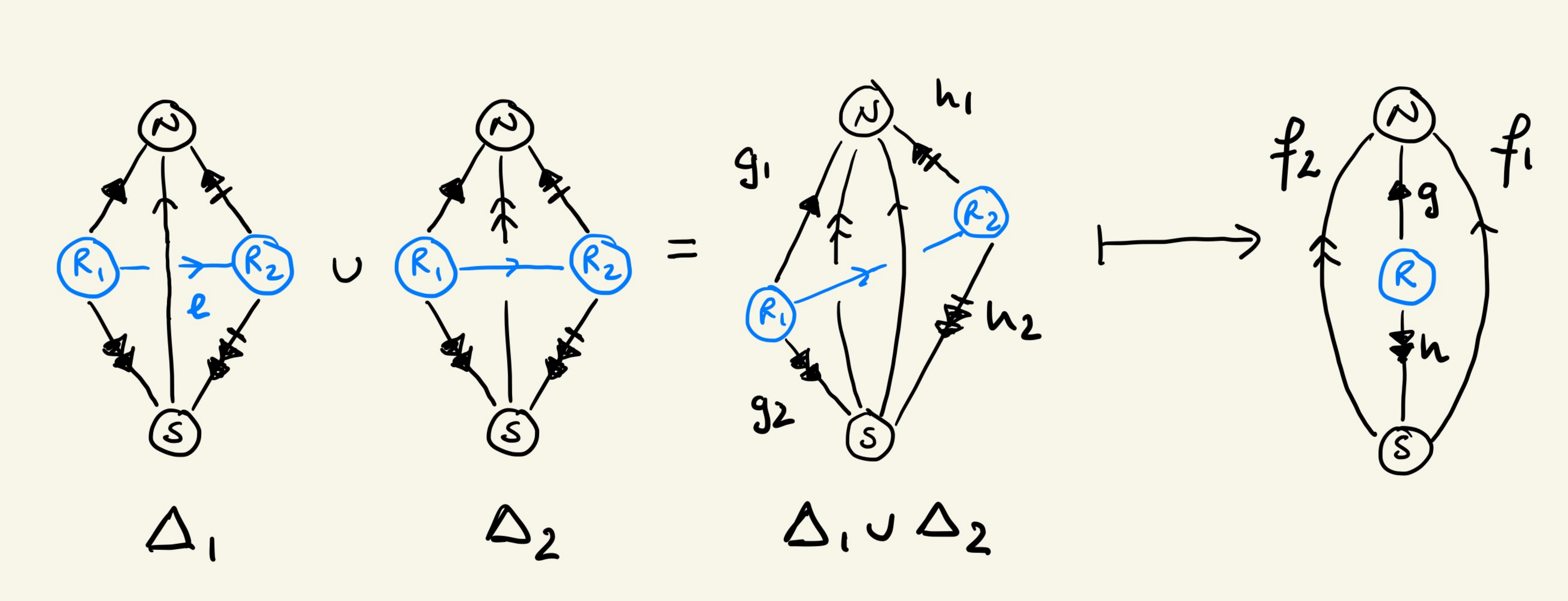}
\caption{The sparse degree-two edge collapse.} 
\label{fig:deg2collapse}
\end{figure}

The operation \textbf{inverse} to the sparse degree-two edge collapse operates on two faces in $\cT$ spanned by three pairwise distinct vertices with the property that the faces share exactly two edges and form an embedded disc in $M$ whose boundary edges have degree at least two. One can choose a normal direction to the interior of the disc in $M$ and perform the inverse operation in a small thickening of the disc in the normal direction. 

The effect of the sparse degree-two edge crush and of its inverse on the complexity pair are respectively:
\begin{align*}
S_{2}:(t,\; v)  &\mapsto (t,\; v) +(-2,\; -1)\\
T_{2}:(t,\; v)  &\mapsto (t,\; v) +(\phantom{-}2,\; \phantom{-}1)
\end{align*}

We can now consider the following decision problem.

\textsc{number of sparse degree-two edge collapses for 3-sphere}. Given a natural number $k$ and two topological triangulations of a given 3--manifold, can we get from one triangulation to the other using at most $k$ sparse degree-two edge collapses?

We remark that the sparse degree-two edge collapse is \emph{somewhat perpendicular} to the $2-0$ edge move defined in \cite{Burton-computational-2013}. With reference to \Cref{fig:deg2collapse}, one can imagine the sparse degree-two edge collapse as performed by flattening $\Delta_1\cup \Delta_2$ with the left hand, collapsing the blue edge connecting $R_1$ and $R_2$ to a single vertex. In contrast, the $2-0$ edge move would be performed by flattening the tetrahedra with the right hand, leaving the blue edge connecting $R_1$ and $R_2$ and instead identifying the two edges connecting $N$ and $S$.

%%%%%%%%%%%%%%%%%%%%%%%%%%%
%%%%%%%%%%%%%%%%%%%%%%%%%%%

\section{The Karp reduction from \textsc{hamiltonian path}}
\label{sec:karp_reduction}

The classical problem \textsc{planar hamiltonian cycle} was proved to be NP--complete by Garey, Johnson, and Tarjan \cite{Garey-planar-1976}. The following modified version was shown to be NP--complete by Koenig and Tsvietkova \cite[Theorem 15]{Koenig-NP-2021}:

 \textsc{modified planar hamiltonian path.} Given planar graph $\Gamma$ with at least two vertices of degree one, find a path that passes through each vertex of $\Gamma$ exactly once.

\vspace{0.05in}

In this section, we construct a Karp reduction of \textsc{modified planar hamiltonian path} to \textsc{number of sparse degree-two edge collapses for 3-sphere}.

%%%%%%%%%%%%%%%%%%%%%%%%%%%

\subsection{Construction of 3--sphere triangulations}
\label{subsec:construction}

A graph is \textbf{simplicial} if it has no one-vertex one-edge loops and if it does not have multiple edges between the same two vertices. 

Let $\Gamma$ be a connected planar simplicial graph and $\varphi \co \Gamma \to S^2$ be an embedding. To simplify notation, we write $\Gamma_\varphi = \varphi(\Gamma).$ The connected components of $S^2 \setminus \Gamma_\varphi$ are called the \textbf{regions} of $\Gamma_\varphi.$ Let $r(\Gamma_\varphi)$ be their total number, and let $v(\Gamma)$ and $e(\Gamma)$ be the number of vertices and edges in $\Gamma.$ Euler's formula gives $2 = v(\Gamma) - e(\Gamma) + r(\Gamma_\varphi),$ and hence the number of regions is independent of the embedding and can be denoted by $r(\Gamma).$ 

We now construct a triangulation $\cT(\Gamma_\varphi)$ of the 3-sphere that only depends on the isotopy class of $\Gamma_\varphi \subset S^2$. The reader may wish to refer to the examples given in \Cref{subsec:examples}.

Let $\Gamma_\varphi^\perp \subset S^2$ be the dual graph to the cell decomposition of $S^2$ determined by $\Gamma_\varphi.$  That is, $\Gamma_\varphi^\perp$ has exactly one vertex in each region of $\Gamma_\varphi$ and one edge dual to each edge of $\Gamma_\varphi.$ We note that $\Gamma_\varphi^\perp$ is planar but not necessarily simplicial.

We claim that each vertex of $\Gamma_\varphi^\perp$ has degree at least two. If $\Gamma$ is a tree, then there is a single vertex in $\Gamma_\varphi^\perp$ and its degree is twice the number of edges in $\Gamma,$ and hence at least two. If $\Gamma$ not a tree, then each region has frontier consisting of at least three edges since $\Gamma$ is simplicial. Hence in this case the degree of each vertex of $\Gamma_\varphi^\perp$ is at least three. 

Embed $S^2 \subset S^3$ as the standard equatorial 2--sphere in the standard 3--sphere imbued with its constant curvature Riemannian metric, and denote the components of $S^3 \setminus S^2$ by $H_N$ and $H_S.$ Choose $N \in H_N$ and $S \in H_S$ as the centroids. So the 3--sphere can be seen as a suspension (or bi-cone) on the equatorial 2-sphere, with vertices $N$ and $S$. In particular, the construction we describe can be given combinatorially by viewing $S^3$ as the suspension on $N$ and $S$ rather than in geometric terms that are more suitable for visualisation.

The $0$--skeleton of $\cT(\Gamma_\varphi)$ is the union of $\{N, S\}$ and the vertex set of $\Gamma_\varphi^\perp.$ Hence there are $2+r(\Gamma)$ vertices in $\cT(\Gamma_\varphi).$

Now connect the vertices of $\Gamma_\varphi$ and $\Gamma_\varphi^\perp$ to both $N$ and $S$ by geodesic segments in $S^3.$ The $1$--skeleton of $\cT(\Gamma_\varphi)$ consists of three three types of edges (see \Cref{fig:edges}):

\begin{figure}[htbp]
\centering
\includegraphics[width=0.5\textwidth]{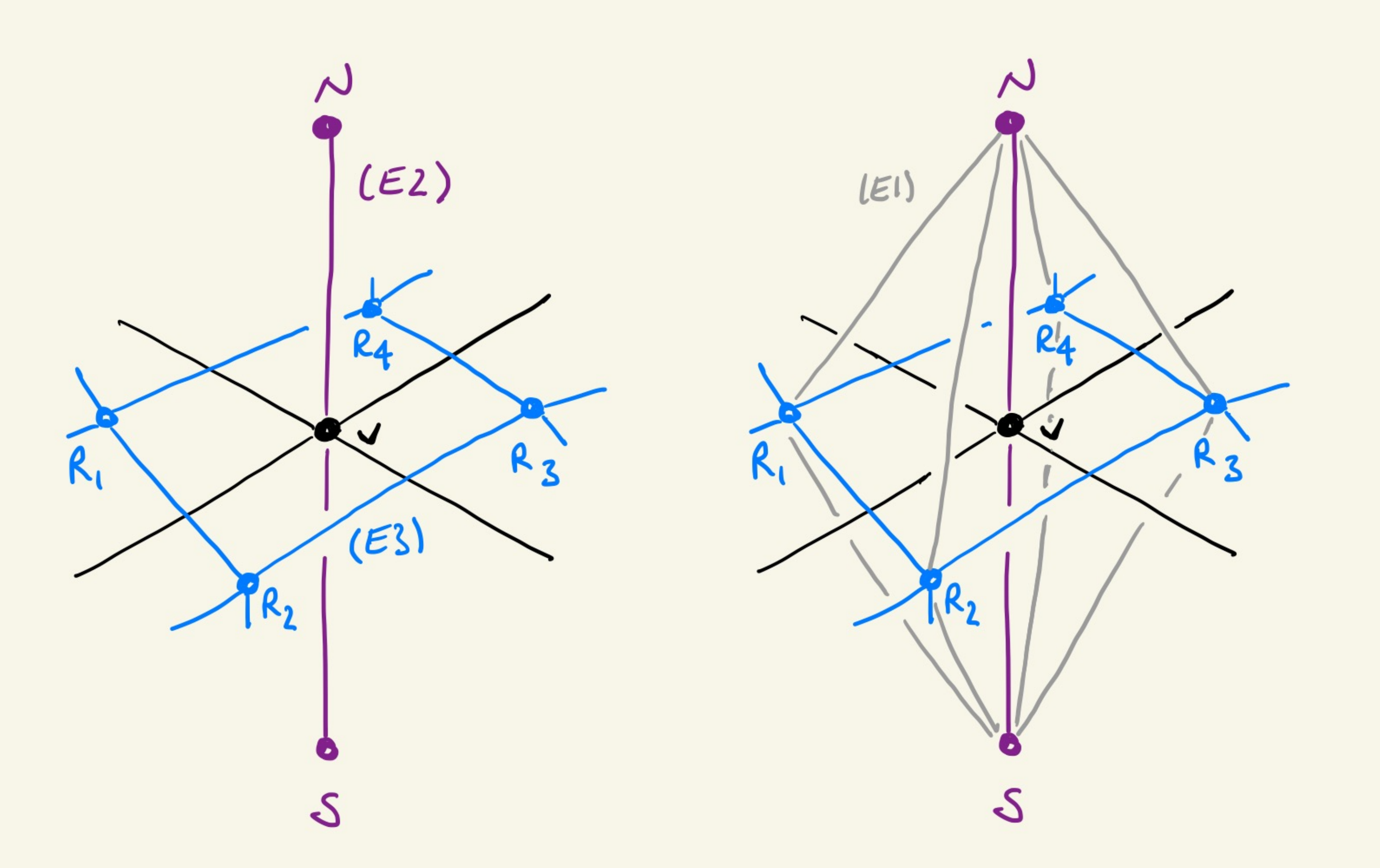}
\caption{Part of $\Gamma_\varphi$ is shown in black, and part of $\Gamma_\varphi^\perp$ in blue. The only edge of type (E2) shown runs from $N$ to $S$ and contains vertex $v$ of $\Gamma_\varphi.$. The edges of $\Gamma_\varphi^\perp$ are of type (E3). The edges connecting vertices of $\Gamma_\varphi^\perp$ to $N$ or $S$ are of type (E1).} 
\label{fig:edges}
\end{figure}

\begin{enumerate}
\item[(E1)] Each segment from a vertex in $\Gamma_\varphi^\perp$ to $N$ (resp. $S$) is an edge in $\cT(\Gamma_\varphi).$ Hence there are $2 r(\Gamma)$ such edges.
\item[(E2)] For each vertex in $\Gamma_\varphi,$ the union of its segments to $N$ and $S$ is an edge in $\cT(\Gamma_\varphi).$ Hence there are $v(\Gamma)$ such edges.
\item[(E3)] Each edge in $\Gamma_\varphi^\perp$ is an edge in $\cT(\Gamma_\varphi)$. Hence there are $e(\Gamma)$ such edges.
\end{enumerate}
It follows that in total, we have $2 r(\Gamma)+v(\Gamma)+e(\Gamma)$ edges in $\cT(\Gamma_\varphi).$ 

\begin{figure}[htbp]
\centering
\includegraphics[width=0.7\textwidth]{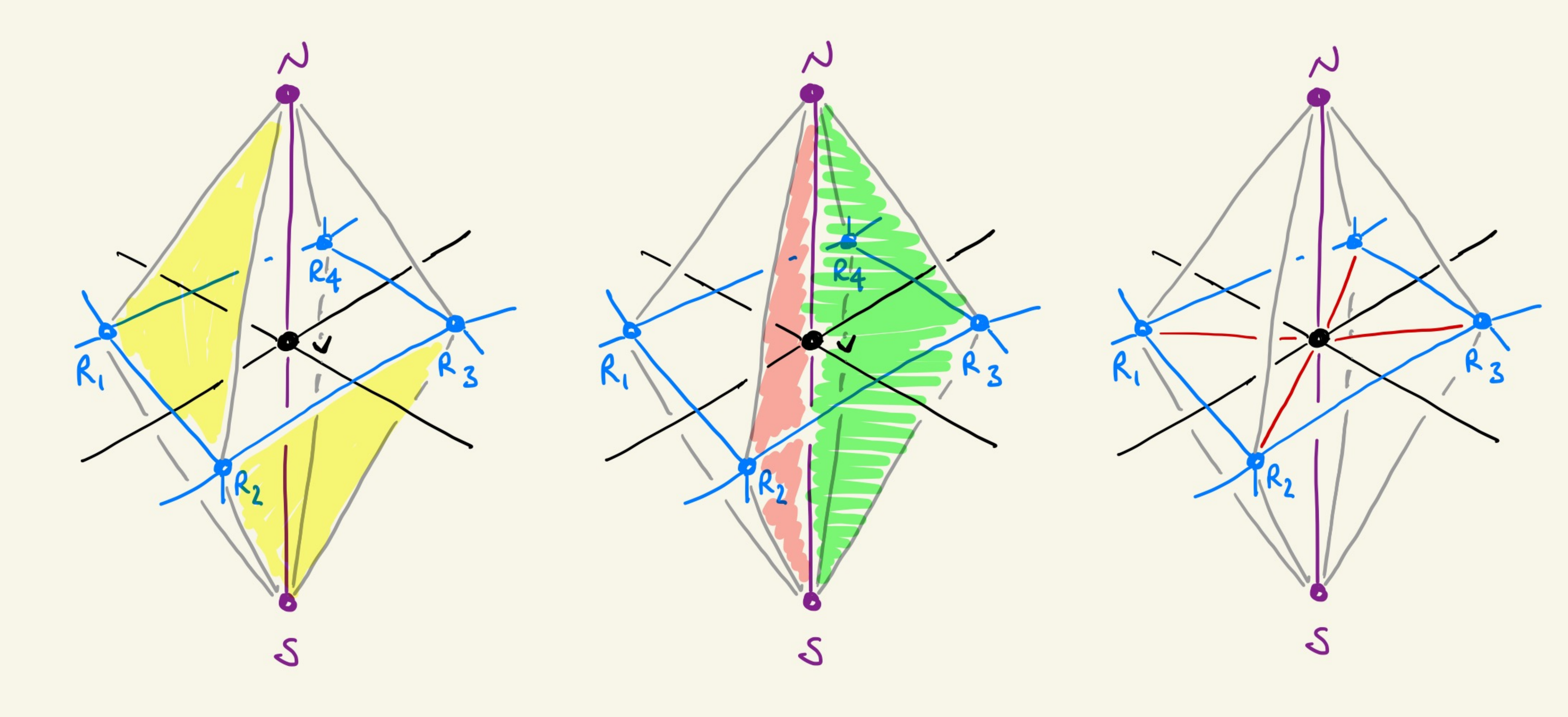}
\caption{Each yellow face shown is a cone on an edge of $\Gamma_\varphi^\perp$ to $N$ or $S.$ The resulting bipyramid in this example is an octahedron. The red and green faces shown result from the subdivision of the bipyramid into tetrahedra meeting around the edge from $N$ to $S.$ The rightmost figure shows the intersection of the faces in the bipyramid with the equatorial sphere.} 
\label{fig:faces}
\end{figure}

Now cone each edge in $\Gamma_\varphi^\perp$ to $N$ (resp. $S$). This gives a set of faces of $\cT(\Gamma_\varphi)$ and decomposes $S^3$ into $v(\Gamma)$ bipyramids with equatorial edges in $\Gamma_\varphi^\perp$ and cone points $N$ and $S.$ Each bipyramid is naturally divided into tetrahedra that are the join of the central edge connecting $N$ and $S$ and an equatorial edge in $\Gamma_\varphi^\perp.$ These are the tetrahedra in $\cT(\Gamma_\varphi).$ It follows from the construction that 
\begin{enumerate}
\item Each edge of type (E1) has degree in $\cT(\Gamma_\varphi)$ twice the degree of the corresponding vertex in $\Gamma^\perp_\varphi.$ Hence its degree is at least four.
\item Each edge of type (E2) has degree in $\cT(\Gamma_\varphi)$ the degree in $\Gamma$ of the corresponding vertex of $\Gamma.$
\item Each edge of type (E3) has degree two in $\cT(\Gamma_\varphi)$. 
\end{enumerate} 
The total number of tetrahedra in $\cT(\Gamma_\varphi)$ is $2e(\Gamma^\perp_\varphi)=2e(\Gamma).$ This implies that the total number of triangles in $\cT(\Gamma_\varphi)$ is $4e(\Gamma).$ 
This completes the construction of $\cT(\Gamma_\varphi).$ As usual, we consider this triangulation up to ambient isotopy of $S^3.$

At this stage, we feel compelled to point out that the above illustrations, whilst correct, do not paint the whole picture as the bipyramids we encounter are usually not embedded in $S^3,$ but may have equatorial vertices identified. The examples in the next section exhibit this behaviour. 

We also mention that different embeddings of a graph may give combinatorially inequivalent triangulations. An example is given in \Cref{fig:inequiv_trigs}. Our main result is independent of the planar embedding chosen for a planar graph.

\begin{figure}[htbp]
\centering
\includegraphics[width=0.5\textwidth]{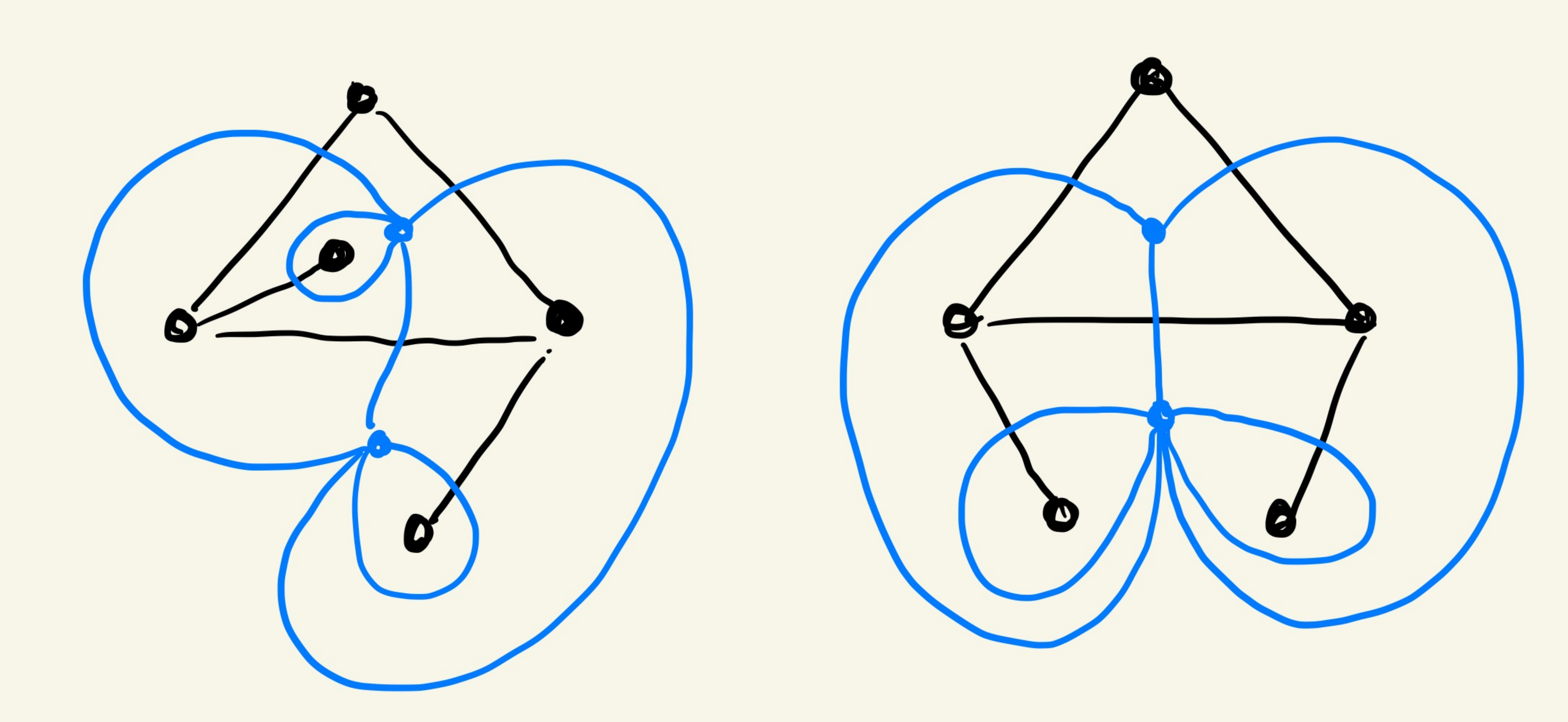}
\caption{Two embeddings of a planar graph (in black), that give inequivalent triangulations. The dual graphs are shown in blue.}
\label{fig:inequiv_trigs}
\end{figure}

%%%%%%%%%%%%%%%%%%%%%%%%%%%

\subsection{Examples}
\label{subsec:examples}

Let $L_n$ be a line segment partitioned into $n$ 1--simplices. Then the isotopy class of the associated triangulation is independent of the embedding of $L_n$ into $S^2$ and hence we write $\cT(L_n).$ Since $L_n$ is a tree, the triangulation $\cT(L_n)$ only has three vertices, denoted $N$, $S$ and $R.$ List the vertices $v_0, \ldots, v_n$ of $L_n$ in linear order. Note that the leaf ending in $v_0$ corresponds to a monogon in the dual graph. In particular, our construction gives a bipyramid on a monogon as shown in the top right of \Cref{fig:L3}. By our construction, this is subdivided into a single tetrahedron (with two faces identified). In particular, the edge of $\cT(L_n)$ corresponding to $v_0$ has degree one. Similarly for $v_n.$ This gives two edges of degree one in $\cT(L_n).$

The $n-1$ edges of $\cT(L_n)$ corresponding to the remaining vertices all have degree two; they are central edges in bipyramids on bigons. The dual graph has one vertex, $R,$ and its $n$ edges correspond to edges of $\cT(L_n)$ that all have degree two. This gives a total of $2n-1$ edges of degree two in $\cT(L_n).$ 

None of the edges of degree two are sparse since the triangulation only has three vertices. There are two more edges in $\cT(L_n)$, one connecting $R$ to $S$ and one connecting $R$ to $N.$ Each of these edges has degree $4n$ since the degree of $R$ in the dual graph is $2n.$ Hence there are $3+2n$ edges in total. For completeness, we record that there are $2n$ tetrahedra in $\cT(L_n):$ one tetrahedron for each of the two vertices of degree one in $L_n$, and two tetrahedra for each of $n-1$ vertices of degree 2. See \Cref{fig:L3} for an example with $n=3$, where we have highlighted a tetrahedron containing the degree-one edge corresponding to $v_0$ and the two tetrahedra containing the degree-two edge corresponding to $v_2.$ 

\begin{figure}[htbp]
\centering
\includegraphics[width=0.9\textwidth]{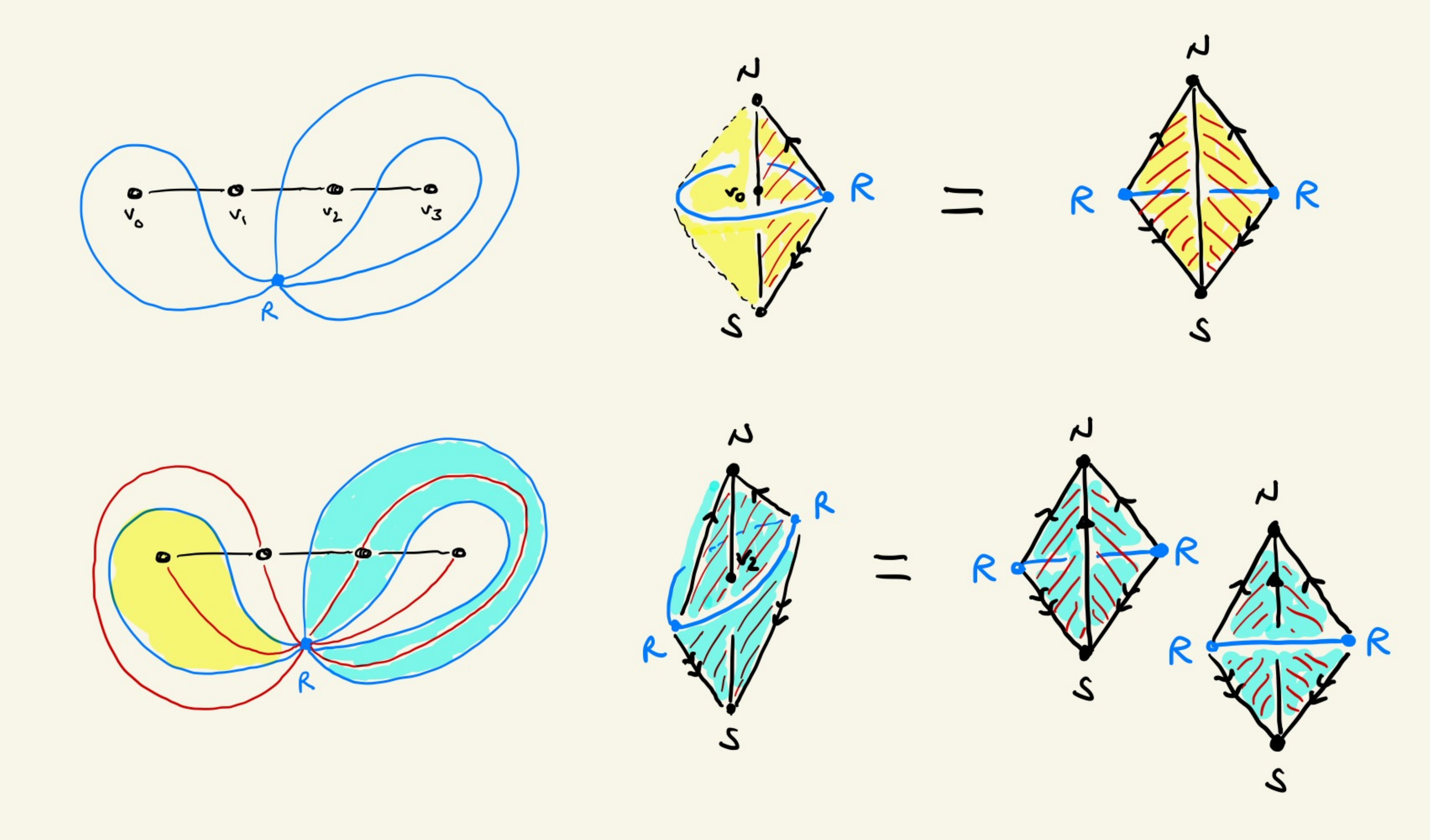}
\caption{Mnemonic of $\cT(L_n)$ for $n=3$. Top center and right: bipyramid based on a monogon, with one-tetrahedron triangulation that has two shaded faces identified together. Bottom center and right: bipyramid based on a bigon, and its decomposition into two tetrahedra. Top left: $\Gamma=L_3\subset S^2$ is drawn in black with vertices $v_i$, and $\Gamma^\perp$ in blue with the single vertex $R$. Bottom left: The intersection of the faces of $\cT(L_3)$ with $S^2$ is indicated in red (as in \Cref{fig:edges}). The edges of $\cT(L_3)$ meet $S_2$ in the edges of $\Gamma^\perp$ and in the vertices of $\Gamma,$ the latter being contained in the interior of edges with endpoints at $S$ and $N.$ The intersection with $S^2$ of the single tetrahedron containing a vertex of degree one of $\Gamma$ is shaded in yellow, and that of the two tetrahedra containing a vertex of degree two in blue. } 
\label{fig:L3}
\end{figure}

%%%%%%%%%%%%%%%%%%%%%%%%%%%

\subsection{Deletion of a non-cut edge with vertex degrees at least three}
\label{subsec:deletion}

We continue with the notation from \Cref{subsec:construction}. Suppose $\Gamma$ has more than one edge. An edge $e$ of $\Gamma$ is \textbf{non-cut} if, after the removal of $e$, we still have a connected graph. Let $\Lambda \subset \Gamma$ be a subgraph that is obtained by deleting a single non-cut edge $e$ with the property that each of its vertices has degree at least three in $\Gamma.$ We denote the restriction of the embedding $\varphi$ to $\Lambda$ also by $\varphi \co \Lambda \to S^2.$ 

Since $e$ is a non-cut edge and not a leaf, $\Lambda$ is connected, $\chi(\Lambda) = \chi(\Gamma) +1$ and $r(\Lambda) = r(\Gamma) -1.$ Indeed, $e$ is adjacent to two distinct regions of $\Gamma_\varphi$, and these have been merged in the complement of $\Lambda_\varphi.$ In terms of the dual graphs, letting $e^\perp \subset \Gamma^\perp_\varphi$ be the edge dual to $e$, we see that $\Lambda^\perp_\varphi$ is obtained from $\Gamma^\perp_\varphi$ by collapsing $e^\perp$ to a point and hence identifying the corresponding endpoints of $e^\perp.$ Note that $e^\perp$ is a degree-two edge in $\cT(\Gamma_\varphi).$ We claim that $\cT(\Gamma_\varphi) \mapsto \cT(\Lambda_\varphi)$ by a sparse edge collapse of $e^\perp.$ \

To see this, denote the tetrahedra containing $e^\perp$ by $\Delta_1$ and $\Delta_2.$ Then the edges of $\Delta_1$ and $\Delta_2$ opposite $e^\perp$ are of type (E2) and correspond to the endpoints of $e.$ In particular, they are distinct and of degree at least 3. In particular, $\Delta_1 \cup \Delta_2$ is incident with four pairwise distinct vertices and form an embedded 3--ball in the 3--sphere. Hence the conditions for a sparse edge collapse of $e^\perp$ are satisfied, and the resulting triangulation is precisely the triangulation $\cT(\Lambda_\varphi).$ Note that on the level of triangulations, two bipyramids that locally meet in two faces incident to $e^\perp$ are transformed into two bipyramids that locally meet only in two edges.

This is illustrated in \Cref{fig:edge_deletion}.

\begin{figure}[htbp]
\centering
\includegraphics[width=0.6\textwidth]{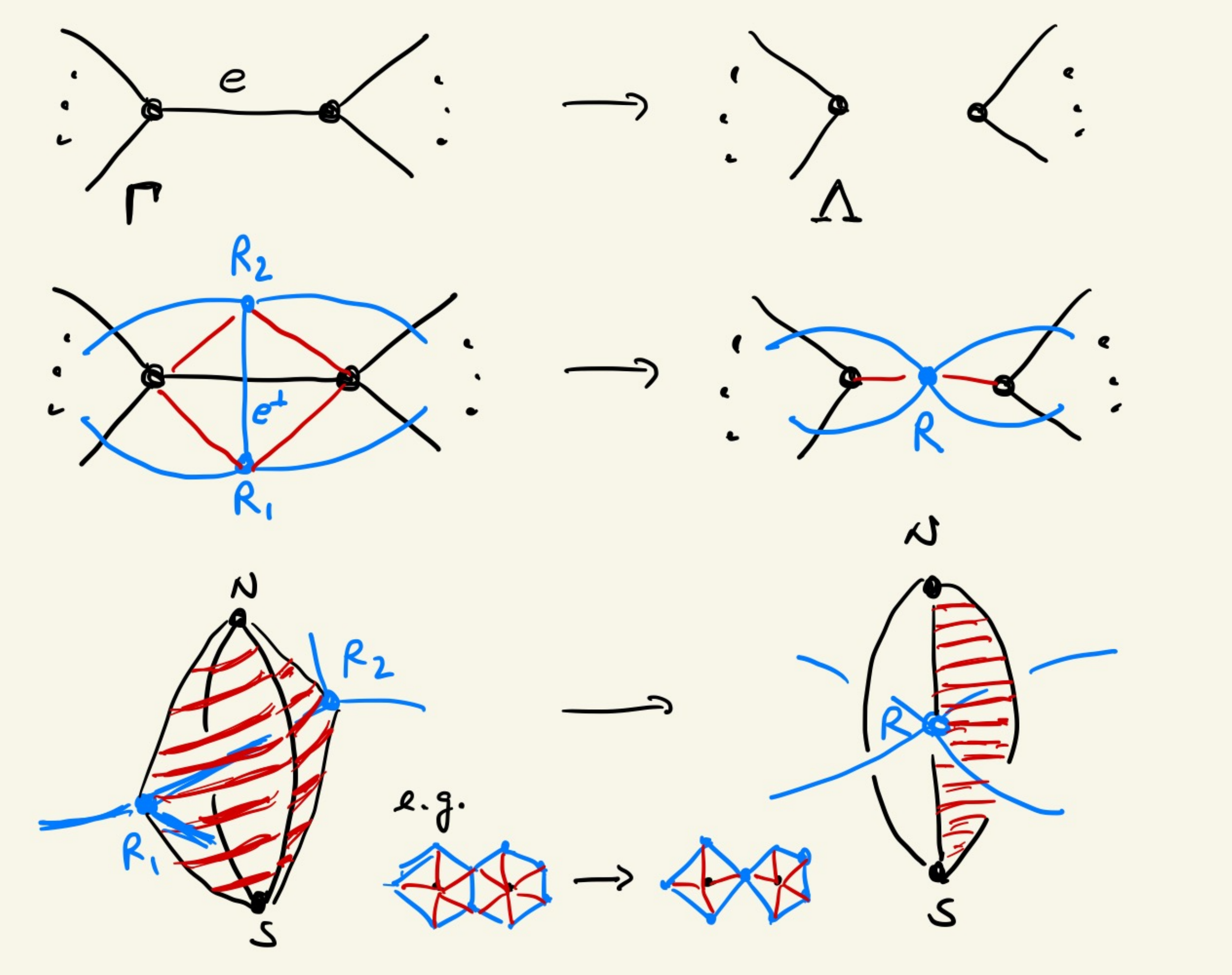}
\caption{Effect of removing a non-cut edge. From top to bottom, shown are portions of $\Gamma$, $\Gamma^\perp_\varphi$ and $\cT(\Gamma_\varphi)$ on the left and the corresponding portions of $\Lambda$, $\Lambda^\perp_\varphi$ and $\cT(\Lambda_\varphi)$ on the right. The effect on the bipyramids associated with the endpoints of the non-cut edge $e$ is that their intersection in $e^\perp$ is changed to an intersection in the vertex $R.$ 
} 
\label{fig:edge_deletion}
\end{figure}

%%%%%%%%%%%%%%%%%%%%%%%%%%%

\subsection{The reduction}
\label{subsec:reduction}

Let $\Gamma$ be a connected simplicial planar graph with at least two degree one vertices. A planar embedding $\varphi\co\Gamma \to S^2$ can be computed in polynomial time in terms of the size of $\Gamma$~\cite{Demoucron-graphes-1964}. Indeed, it can be computed in linear time in the number of vertices of $\Gamma$~\cite{Chiba-linear-1985}. Hence given, as input, a planar graph $\Gamma,$ as a first step in our construction we compute an embedding $\varphi\co \Gamma \to S^2$ using~\cite{Chiba-linear-1985}. The $2e(\Gamma)$ tetrahedra and $16 e(\Gamma)$ face pairings of the triangulation $\cT_1 = \cT(\Gamma_\varphi)$ can then be constructed in polynomial time in terms of $v(\Gamma)+e(\Gamma)$ from $\varphi.$ 

Also construct the triangulation $\cT_2 = \cT(L_{v(\Gamma)-1})$ from \Cref{subsec:examples}. Since 
\[
e(L_{v(\Gamma)-1}) = v(\Gamma)-1 = e(\Gamma) + 1 -r(\Gamma) \le e(\Gamma)
\]
it follows that $\cT_2$ can also be computed in polynomial time in terms of $v(\Gamma)+e(\Gamma)$ and that the number of tetrahedra in $\cT_1$ is greater or equal to the number of tetrahedra in $\cT_2.$

Let $\mathcal{C}_2$ be the set of sparse degree-two edge collapses or their inverses.

\begin{proposition}\label{claim:collapses_iff_hamiltonian}
There is a Hamiltonian path $C$ in $\Gamma$ if and only if the triangulations $\cT_1$ and $\cT_2$ are related by a sequence of $k=e(\Gamma)-v(\Gamma)+1$ moves in $\mathcal{C}_2$, and all $k$ moves are either sparse degree-two edge collapses or all $k$ moves are their inverses. 
\end{proposition}

\begin{proof}
Suppose there is a Hamiltonian path $C$ in $\Gamma.$ This necessarily starts and ends at the two degree one vertices. The number of edges in $C$ is $v(\Gamma)-1$. One can now iteratively delete the edges in $\Gamma \setminus C,$ giving a sequence of graphs 
\[
\Gamma \supset \Gamma_1 \supset \Gamma_2 \supset \ldots \supset \Gamma_{e(\Gamma)-v(\Gamma)+1} = C
\]
Note that an edge deleted from $\Gamma_i$ in this sequence is a non-cut edge that has no degree one or two vertices with respect to $\Gamma_i$ because $\Gamma_i$ contains the Hamiltonian path $C.$ Hence, as explained in \Cref{subsec:deletion}, one has an associated sequence of sparse degree-two edge collapses connecting the two triangulations (where we have suppressed the reference to the embedding $\varphi$, which is restricted to each subgraph):
\[
\cT(\Gamma) \mapsto \cT(\Gamma_1) \mapsto \cT(\Gamma_2) \mapsto \ldots \mapsto \cT(\Gamma_{e(\Gamma)-v(\Gamma)+1}) = \cT(C)
\]
For the converse, since a sparse degree-two edge collapse decreases the number of tetrahedra, we may suppose that we can obtain $\cT_2$ from $\cT_1$ by $k=e(\Gamma)-v(\Gamma)+1$ sparse degree-two edge collapses. If we can show that no sparse degree-two edge collapse involves an edge that is incident with $N$ or $S$, then we can relate the edge collapsing sequence to a sequence of deletions of edges of $\Gamma$ that result in $C.$

Note that each edge collapse identifies two vertices. The triangulation $\cT(\Gamma_\varphi)$ has $2+r(\Gamma)$ vertices and the triangulation $\cT(C)$ has three vertices, which we denote $N_C, S_C, R_C$ according to whether the vertex corresponds to $N$, $S$ or the single region of $C.$ There are $v(\Gamma)$ edges between $N_C$ and $S_C$, and each of $N_C$ and $S_C$ also has one edge connecting it to $R_C.$ The vertex $R_C$ in addition has $v(\Gamma)-1$ loops based at it. This is the complete set of edges in $\cT(C).$

We use analogous notation for the vertices of $\cT(\Gamma_\varphi).$ Note that there are also exactly $v(\Gamma)$ edges between $N_\Gamma$ and $S_\Gamma,$ and that this number does not change when edges in $\Gamma_\varphi^\perp$ are collapsed.

Suppose there is a sequence of sparse degree-two edge collapses
\[
\cT(\Gamma) \mapsto \cT_1 \mapsto \cT_2 \mapsto \ldots \mapsto \cT_{e(\Gamma)-v(\Gamma)+1} = \cT(C)
\]
If $\cT(\Gamma) \mapsto \cT_1$ is a sparse degree-two edge collapse of an edge $e^\perp$ in $\Gamma_\varphi^\perp$, then we may identify $\cT_1 = \cT(\Gamma_1)$ for the subgraph $\Gamma_1 \subset \Gamma$ obtained by deleting the edge $e$ of $\Gamma$ dual to $e^\perp.$ Since the degree-two edge is sparse, we also note each endpoint of $e$ has degree at least three and that $e$ is a non-cut edge.
This process can be continued inductively. If $\cT(\Gamma_i) = \cT_i \mapsto \cT_{i+1}$ is a sparse degree-two edge collapse of an edge $e_i^\perp$ in $\Gamma_i^\perp$, then we may identify $\cT_{i+1} = \cT(\Gamma_{i+1})$ for the subgraph $\Gamma_{i+1} \subset \Gamma_i$ obtained by deleting the edge $e_i$ of $\Gamma_i$ dual to $e_i^\perp.$ Moreover, each endpoint of $e_i$ has degree at least three and $e_i$ is a non-cut edge.

Suppose that this process continues until we arrive at $\cT(\Gamma_{e(\Gamma)-v(\Gamma)+1}) =  \cT(C).$ Hence deleting $e(\Gamma)-v(\Gamma)+1$ edges from $\Gamma$ gives the graph $\Gamma_{e(\Gamma)-v(\Gamma)+1} \subset \Gamma.$  We note that $\Gamma_{e(\Gamma)-v(\Gamma)+1}$ contains all vertices of $\Gamma,$ and hence $r(\Gamma_{e(\Gamma)-v(\Gamma)+1})=1$ by Euler's formula. Whence $\Gamma_{e(\Gamma)-v(\Gamma)+1}$ is a tree. 
Since each deleted edge has vertex degrees at least three and is a non-cut edge, $\Gamma_{e(\Gamma)-v(\Gamma)+1}$ has exactly two leaves. Hence $\Gamma_{e(\Gamma)-v(\Gamma)+1}$ is Hamiltonian path in $\Gamma.$ 

To complete the proof, suppose there is a triangulation $\cT_i = \cT(\Gamma_i)$ in which we perform a sparse degree-two edge collapse that is not of an edge in $\Gamma_i^\perp.$ Since each edge of type (E1) has degree at least four, this must be a collapse of a degree two edge $e$ between $N_\Gamma$ and $S_\Gamma.$ However the edges opposite $e$ in the two tetrahedra containing it have degree two since they are contained in $\Gamma_i^\perp$. But then $e$ is not a sparse degree-two edge. 
Hence there is no collapse of an edge between $N_\Gamma$ and $S_\Gamma.$ It follows that a sequence of $e(\Gamma)-v(\Gamma)+1$ sparse degree-two edge collapses between $\cT_1$ and $\cT_2$ detects a successive deletion of edges from $\Gamma$ that results in a Hamiltonian path in $\Gamma.$
\end{proof}

\begin{theorem} 
\label{thm:sparse_is_NP-complete}
\textsc{number of sparse degree-two edge collapses for 3-sphere} is NP-hard.
\end{theorem}

\begin{proof}
Given a planar graph $\Gamma,$ we have already argued that the triangulations $\cT_1$ and $\cT_2$ are constructed in time polynomial in $v(\Gamma)+e(\Gamma)$ Hence the reduction is polynomial. \Cref{claim:collapses_iff_hamiltonian} therefore implies we have a Karp reduction as stated at the beginning of \Cref{sec:karp_reduction}.
\end{proof}

The above problem (or the equivalent problem involving bistellar moves) is likely not in NP. An explicit exponential-type upper bound on the number of \emph{bistellar moves} needed to pass between two triangulations of the 3--sphere was obtained by Mijatovi\'{c}~\cite{Mijatovic-simplifying-2003}.

%%%%%%%%%%%%%%%%%%%%%%%%%%%
%%%%%%%%%%%%%%%%%%%%%%%%%%%

\begin{proof}[Proof of \Cref{thm:main}]
We are required to include the set $\mathcal{B}$ of all bistellar moves between triangulations in addition to sparse degree-two edge collapses and their inverses. Appealing to \Cref{claim:collapses_iff_hamiltonian} and \Cref{thm:sparse_is_NP-complete}, it suffices to show that a minimal sequence of elementary moves in $\mathcal{B}\cup \mathcal{C}_2$ required to pass from $\cT_1$ to $\cT_2$ consists only of sparse degree-two edge collapses. 

We know that we can pass from $\cT_1$ to $\cT_2$ with $k=e(\Gamma)-v(\Gamma)+1$ sparse degree-two edge collapses. To simplify notation, let $n = v(\Gamma)-1.$ Then $\cT_2 = \cT(L_n).$

Suppose there is a sequence of moves from $\mathcal{B}\cup \mathcal{C}_2$ that transforms $\cT_1$ into $\cT_2$ and which takes at most $k$ moves. Assume this sequence consists of $a$ 1--4 moves, $b$ 2--3 moves, $c$ 3--2 moves, $d$ 4--1 moves, $e$ sparse degree-two edge collapses and $f$ of their inverses. Then 
\begin{equation}\label{eq:weight}
a+b+c+d+e+f \le k
\end{equation}
The complexity pair of $\cT_2$ is
\[
c(\cT_2) = (2n, 3)
\]
Since $\cT_2$ is obtained from $\cT_1$ by $k$ sparse degree-two edge collapses, we have 
\[
c(\cT_1) = (2n+2k, 3+k)
\]

Hence 
\begin{align*}
(2n, 3)&= c(\cT_2)  \\
&= c(\cT_1)  + a (3,1) + b (1, 0) + c (-1,0)+ d(-3, -1) + e(-2, -1) + f (2, 1)\\
		&= (2n+2k, 3+k) + a (3,1) + b (1, 0) + c (-1,0)+ d(-3, -1) + e(-2, -1) + f (2, 1)
\end{align*}
This gives the equations
\begin{align*}
0 &= 3 a + b - c - 3 d - 2 e + 2 f + 2 k\\
0 &= a - d - e + f + k
\end{align*}
Solving the second for $e$ and substituting into the first gives:
\begin{align*}
e &= a - d + f + k\\
a+b &= c+d
\end{align*}
Also, substituting the expression for $e$ into \Cref{eq:weight} gives:
\begin{equation*}
2a+b+c+ 2f  \le 0
\end{equation*}
This forces $a=b=c=f=0$ since all terms are non-negative, and hence we obtain $d=0$ and $e=k.$ Hence even if one allows an arbitrary sequence of elementary moves from the larger set $\mathcal{B}\cup \mathcal{C}_2$, the minimal sequence remains a sequence of sparse degree-two edge crushes. 
\end{proof}

\begin{scholion}
The minimal number of moves in $\mathcal{B}$ required to pass from $\cT_1$ to $\cT_2$ is at least \[2k=2e(\Gamma)-2v(\Gamma)+2\]
\end{scholion}
\begin{proof}
Similar to above, using the smaller set of moves one obtains
\begin{align*}
(2n, 3)&= c(\cT_2)  \\
&= c(\cT_1)  + a (3,1) + b (1, 0) + c (-1,0)+ d(-3, -1) \\
		&= (2n+2k, 3+k) + a (3,1) + b (1, 0) + c (-1,0)+ d(-3, -1) 
\end{align*}
This simplifies to the equations
\begin{align*}
b &=  c  + k \ge k\\
d &= a  + k \ge k
\end{align*}
Hence a minimal sequence necessarily includes at least $k$ 2--3 moves and $k$ 4--1 moves. 
\end{proof}

%%%%%%%%%%%%%%%%%%%%%%%%%%%

\bibliographystyle{plain}
\bibliography{connectivity}

@unpublished{Rieck,
  author       = {Cheng, Shannon and Chlopecki, Anna and  Nazar, Saarah and  Samperton, Eric},
  title        = {An elementary proof that linking problems are hard},
  year         = {preprint},
  note         = {https://arxiv.org/abs/2509.13120}
}

@article {Mijatovic-simplifying-2003,
    AUTHOR = {Mijatovi\'{c}, Aleksandar},
     TITLE = {Simplifying triangulations of {$S^3$}},
   JOURNAL = {Pacific J. Math.},
  FJOURNAL = {Pacific Journal of Mathematics},
    VOLUME = {208},
      YEAR = {2003},
    NUMBER = {2},
     PAGES = {291--324},
      ISSN = {0030-8730,1945-5844},
   MRCLASS = {52B70 (57Q15)},
  MRNUMBER = {1971667},
       DOI = {10.2140/pjm.2003.208.291},
       URL = {https://doi.org/10.2140/pjm.2003.208.291},
}

@article {Joswig-computing-2006,
    AUTHOR = {Joswig, Michael and Pfetsch, Marc E.},
     TITLE = {Computing optimal {M}orse matchings},
   JOURNAL = {SIAM J. Discrete Math.},
  FJOURNAL = {SIAM Journal on Discrete Mathematics},
    VOLUME = {20},
      YEAR = {2006},
    NUMBER = {1},
     PAGES = {11--25},
      ISSN = {0895-4801,1095-7146},
   MRCLASS = {90C27 (57Q05 68Q17 90C57)},
  MRNUMBER = {2257241},
MRREVIEWER = {Winfried\ Hochst\"{a}ttler},
       DOI = {10.1137/S0895480104445885},
       URL = {https://doi.org/10.1137/S0895480104445885},
}

@article {Lackenby-recognising-2026,
    AUTHOR = {Lackenby, Marc and Schleimer, Saul},
     TITLE = {Recognising elliptic manifolds},
   JOURNAL = {Comment. Math. Helv.},
  FJOURNAL = {Commentarii Mathematici Helvetici. A Journal of the Swiss
              Mathematical Society},
    VOLUME = {101},
      YEAR = {2026},
    NUMBER = {1},
     PAGES = {1--45},
      ISSN = {0010-2571,1420-8946},
   MRCLASS = {57Q15 (57K30 57K35 57Z25 68Q25)},
  MRNUMBER = {5032283},
       DOI = {10.4171/cmh/597},
       URL = {https://doi.org/10.4171/cmh/597},
}

@article {Jackson-recognition-2025,
    AUTHOR = {Jackson, Adele},
     TITLE = {Recognition of {S}eifert fibered spaces with boundary is in
              {NP}},
   JOURNAL = {Math. Ann.},
  FJOURNAL = {Mathematische Annalen},
    VOLUME = {391},
      YEAR = {2025},
    NUMBER = {1},
     PAGES = {309--361},
      ISSN = {0025-5831,1432-1807},
   MRCLASS = {57K30 (57-08 57K35 57Q15 68Q25)},
  MRNUMBER = {4846785},
MRREVIEWER = {Jennifer\ Schultens},
       DOI = {10.1007/s00208-024-02920-x},
       URL = {https://doi.org/10.1007/s00208-024-02920-x},
}

@article {Chiba-linear-1985,
    AUTHOR = {Chiba, Norishige and Nishizeki, Takao and Abe, Shigenobu and
              Ozawa, Takao},
     TITLE = {A linear algorithm for embedding planar graphs using
              {$PQ$}-trees},
   JOURNAL = {J. Comput. System Sci.},
  FJOURNAL = {Journal of Computer and System Sciences},
    VOLUME = {30},
      YEAR = {1985},
    NUMBER = {1},
     PAGES = {54--76},
      ISSN = {0022-0000},
   MRCLASS = {05C10 (05C05 68Q20)},
  MRNUMBER = {788831},
MRREVIEWER = {R.\ C.\ Read},
       DOI = {10.1016/0022-0000(85)90004-2},
       URL = {https://doi.org/10.1016/0022-0000(85)90004-2},
}

@unpublished{HT,
  author       = {Touseef, Haider and Tsvietkova, Anastasiia},
  title        = {Polynomial algorithm for alternating link equivalence},
  year         = {preprint},
  note         = {https://arxiv.org/abs/2412.02003}
}

@article {Ivanov,
    AUTHOR = {Ivanov, S. V.},
     TITLE = {The computational complexity of basic decision problems in
              3-dimensional topology},
   JOURNAL = {Geom. Dedicata},
  FJOURNAL = {Geometriae Dedicata},
    VOLUME = {131},
      YEAR = {2008},
     PAGES = {1--26},
      ISSN = {0046-5755,1572-9168},
   MRCLASS = {57M40 (57M05 57M50 57N10 68Q25)},
  MRNUMBER = {2369189},
MRREVIEWER = {Wolfgang\ H.\ Heil},
       DOI = {10.1007/s10711-007-9210-4},
       URL = {https://doi-org.proxy.libraries.rutgers.edu/10.1007/s10711-007-9210-4},
}

@incollection {Schleimer,
    AUTHOR = {Schleimer, Saul},
     TITLE = {Sphere recognition lies in {NP}},
 BOOKTITLE = {Low-dimensional and symplectic topology},
    SERIES = {Proc. Sympos. Pure Math.},
    VOLUME = {82},
     PAGES = {183--213},
 PUBLISHER = {Amer. Math. Soc., Providence, RI},
      YEAR = {2011},
      ISBN = {978-0-8218-5235-4},
   MRCLASS = {57M40 (57-04)},
  MRNUMBER = {2768660},
MRREVIEWER = {Bruno\ P.\ Zimmermann},
       DOI = {10.1090/pspum/082/2768660},
       URL = {https://doi-org.proxy.libraries.rutgers.edu/10.1090/pspum/082/2768660},
}

@book {K3,
    AUTHOR = {Baykur, R. İnanç and Kirby, Robion C. and  Ruberman, Daniel},
     TITLE = {K3: A New Problem List in Low-Dimensional Topology},
    SERIES = {in press},
      NOTE = {},
 PUBLISHER = {American Mathematical Society (AMS)},
      YEAR = {2026},
     PAGES = {},
      ISBN = {},
   MRCLASS = {},
  MRNUMBER = {},
MRREVIEWER = {},
       DOI = {},
       URL = {},
}

@inproceedings {Babai,
    AUTHOR = {Babai, L\'aszl\'o},
     TITLE = {Graph isomorphism in quasipolynomial time [extended abstract]},
 BOOKTITLE = {S{TOC}'16---{P}roceedings of the 48th {A}nnual {ACM} {SIGACT}
              {S}ymposium on {T}heory of {C}omputing},
     PAGES = {684--697},
 PUBLISHER = {ACM, New York},
      YEAR = {2016},
      ISBN = {978-1-4503-4132-5},
   MRCLASS = {68R10 (05C60 05C85 20B05 68Q25)},
  MRNUMBER = {3536606},
       DOI = {10.1145/2897518.2897542},
       URL = {https://doi.org/10.1145/2897518.2897542},
}

@article {ScottShort,
    AUTHOR = {Scott, Peter and Short, Hamish},
     TITLE = {The homeomorphism problem for closed 3-manifolds},
   JOURNAL = {Algebr. Geom. Topol.},
  FJOURNAL = {Algebraic \& Geometric Topology},
    VOLUME = {14},
      YEAR = {2014},
    NUMBER = {4},
     PAGES = {2431--2444},
      ISSN = {1472-2747,1472-2739},
   MRCLASS = {57M50 (20F65)},
  MRNUMBER = {3331689},
MRREVIEWER = {Rafael\ Oswaldo\ Ruggiero},
       DOI = {10.2140/agt.2014.14.2431},
       URL = {https://doi.org/10.2140/agt.2014.14.2431},
}

@article {Kuperberg,
    AUTHOR = {Kuperberg, Greg},
     TITLE = {Algorithmic homeomorphism of 3-manifolds as a corollary of
              geometrization},
   JOURNAL = {Pacific J. Math.},
  FJOURNAL = {Pacific Journal of Mathematics},
    VOLUME = {301},
      YEAR = {2019},
    NUMBER = {1},
     PAGES = {189--241},
      ISSN = {0030-8730,1945-5844},
   MRCLASS = {57M50 (57K30 68Q15)},
  MRNUMBER = {4007377},
MRREVIEWER = {Jessica\ S.\ Purcell},
       DOI = {10.2140/pjm.2019.301.189},
       URL = {https://doi.org/10.2140/pjm.2019.301.189},
}

@inproceedings {BurtonSpreer,
    AUTHOR = {Burton, Benjamin A. and Spreer, Jonathan},
     TITLE = {The complexity of detecting taut angle structures on
              triangulations},
 BOOKTITLE = {Proceedings of the {T}wenty-{F}ourth {A}nnual {ACM}-{SIAM}
              {S}ymposium on {D}iscrete {A}lgorithms},
     PAGES = {168--183},
 PUBLISHER = {SIAM, Philadelphia, PA},
      YEAR = {2012},
      ISBN = {978-1-611972-52-8},
   MRCLASS = {68U05 (57M50 68Q25)},
  MRNUMBER = {3185388},
}

@article {deMesmayRieckTancer,
    AUTHOR = {de Mesmay, Arnaud and Rieck, Yo'av and Sedgwick, Eric and
              Tancer, Martin},
     TITLE = {Embeddability in {$\Bbb R^3$} is {NP}-hard},
   JOURNAL = {J. ACM},
  FJOURNAL = {Journal of the ACM},
    VOLUME = {67},
      YEAR = {2020},
    NUMBER = {4},
     PAGES = {Art. 20, 29},
      ISSN = {0004-5411,1557-735X},
   MRCLASS = {57R40 (68Q17 68U05)},
  MRNUMBER = {4116340},
MRREVIEWER = {Sang-Eon\ Han},
       DOI = {10.1145/3396593},
       URL = {https://doi.org/10.1145/3396593},
}

@article {BurtondeMesmayWagner,
    AUTHOR = {Burton, Benjamin A. and de Mesmay, Arnaud and Wagner, Uli},
     TITLE = {Finding non-orientable surfaces in 3-manifolds},
   JOURNAL = {Discrete Comput. Geom.},
  FJOURNAL = {Discrete \& Computational Geometry. An International Journal
              of Mathematics and Computer Science},
    VOLUME = {58},
      YEAR = {2017},
    NUMBER = {4},
     PAGES = {871--888},
      ISSN = {0179-5376,1432-0444},
   MRCLASS = {68U05 (57M50 68Q25)},
  MRNUMBER = {3717241},
MRREVIEWER = {Hidetoshi\ Masai},
       DOI = {10.1007/s00454-017-9900-0},
       URL = {https://doi.org/10.1007/s00454-017-9900-0},
}

@inproceedings {KoenigTsvietkova1,
    AUTHOR = {Koenig, Dale and Tsvietkova, Anastasiia},
     TITLE = {Unlinking, splitting, and some other {NP}-hard problems in
              knot theory},
 BOOKTITLE = {Proceedings of the 2021 {ACM}-{SIAM} {S}ymposium on {D}iscrete
              {A}lgorithms ({SODA})},
     PAGES = {1496--1507},
 PUBLISHER = {[Society for Industrial and Applied Mathematics (SIAM)],
              Philadelphia, PA},
      YEAR = {2021},
      ISBN = {978-1-61197-646-5},
   MRCLASS = {68Q17 (57K10)},
  MRNUMBER = {4262523},
       DOI = {10.1137/1.9781611976465.90},
       URL = {https://doi.org/10.1137/1.9781611976465.90},
}

@article {AgolHassThurston,
    AUTHOR = {Agol, Ian and Hass, Joel and Thurston, William},
     TITLE = {The computational complexity of knot genus and spanning area},
   JOURNAL = {Trans. Amer. Math. Soc.},
  FJOURNAL = {Transactions of the American Mathematical Society},
    VOLUME = {358},
      YEAR = {2006},
    NUMBER = {9},
     PAGES = {3821--3850},
      ISSN = {0002-9947,1088-6850},
   MRCLASS = {68Q17 (11Y16 57M25 57M50)},
  MRNUMBER = {2219001},
MRREVIEWER = {Mario\ Eudave Mu\~noz},
       DOI = {10.1090/S0002-9947-05-03919-X},
       URL = {https://doi.org/10.1090/S0002-9947-05-03919-X},
}

@incollection {LackenbyKnotTheory,
    AUTHOR = {Lackenby, Marc},
     TITLE = {Elementary knot theory},
 BOOKTITLE = {Lectures on geometry},
    SERIES = {Clay Lect. Notes},
     PAGES = {29--64},
 PUBLISHER = {Oxford Univ. Press, Oxford},
      YEAR = {2017},
      ISBN = {978-0-19-878491-3},
   MRCLASS = {57M25},
  MRNUMBER = {3676592},
MRREVIEWER = {Colin\ C.\ Adams},
}

@article {LackenbyNPhard,
    AUTHOR = {Lackenby, Marc},
     TITLE = {Some conditionally hard problems on links and 3-manifolds},
   JOURNAL = {Discrete Comput. Geom.},
  FJOURNAL = {Discrete \& Computational Geometry. An International Journal
              of Mathematics and Computer Science},
    VOLUME = {58},
      YEAR = {2017},
    NUMBER = {3},
     PAGES = {580--595},
      ISSN = {0179-5376,1432-0444},
   MRCLASS = {57M25 (57N10 68Q17 68Q25)},
  MRNUMBER = {3690662},
MRREVIEWER = {Jonathan\ Spreer},
       DOI = {10.1007/s00454-017-9905-8},
       URL = {https://doi.org/10.1007/s00454-017-9905-8},
}

@article {deMesmayRieckSedgwickTancer,
    AUTHOR = {de Mesmay, Arnaud and Rieck, Yo'av and Sedgwick, Eric and
              Tancer, Martin},
     TITLE = {The unbearable hardness of unknotting},
   JOURNAL = {Adv. Math.},
  FJOURNAL = {Advances in Mathematics},
    VOLUME = {381},
      YEAR = {2021},
     PAGES = {Paper No. 107648, 36},
      ISSN = {0001-8708,1090-2082},
   MRCLASS = {57K10 (68Q17)},
  MRNUMBER = {4215750},
MRREVIEWER = {Mikhail\ V.\ Volkov},
       DOI = {10.1016/j.aim.2021.107648},
       URL = {https://doi.org/10.1016/j.aim.2021.107648},
}

@article {deMesmaySchaeferSedgwick,
    AUTHOR = {de Mesmay, Arnaud and Schaefer, Marcus and Sedgwick, Eric},
     TITLE = {Link crossing number is {NP}-hard},
   JOURNAL = {J. Knot Theory Ramifications},
  FJOURNAL = {Journal of Knot Theory and its Ramifications},
    VOLUME = {29},
      YEAR = {2020},
    NUMBER = {6},
     PAGES = {2050043, 15},
      ISSN = {0218-2165,1793-6527},
   MRCLASS = {57K10 (57M15 68Q17 68R10)},
  MRNUMBER = {4125179},
MRREVIEWER = {Davide\ Bil\`o},
       DOI = {10.1142/S0218216520500431},
       URL = {https://doi.org/10.1142/S0218216520500431},
}

@misc{regina,
    author = {Benjamin A. Burton and Ryan Budney and William Pettersson and others},
    title = {Regina: Software for low-dimensional topology},
    howpublished = {{\tt http://\allowbreak regina-normal.\allowbreak github.\allowbreak io/}},
    year = {1999--2026}}

@incollection {Lackenby-algorithms-2022,
    AUTHOR = {Lackenby, Marc},
     TITLE = {Algorithms in 3-manifold theory},
 BOOKTITLE = {Surveys in differential geometry 2020. {S}urveys in 3-manifold
              topology and geometry},
    SERIES = {Surv. Differ. Geom.},
    VOLUME = {25},
     PAGES = {163--213},
 PUBLISHER = {Int. Press, Boston, MA},
      YEAR = {[2022] \copyright 2022},
      ISBN = {978-1-57146-419-4},
   MRCLASS = {57K10 (57M50 68Q25 68W40)},
  MRNUMBER = {4479752},
}

@article {Koenig-NP-2021,
    AUTHOR = {Koenig, Dale and Tsvietkova, Anastasiia},
     TITLE = {N{P}-hard problems naturally arising in knot theory},
   JOURNAL = {Trans. Amer. Math. Soc. Ser. B},
  FJOURNAL = {Transactions of the American Mathematical Society. Series B},
    VOLUME = {8},
      YEAR = {2021},
     PAGES = {420--441},
      ISSN = {2330-0000},
   MRCLASS = {57K10 (57-08)},
  MRNUMBER = {4273193},
MRREVIEWER = {Colin\ C.\ Adams},
       DOI = {10.1090/btran/71},
       URL = {https://doi.org/10.1090/btran/71},
}

@article {Garey-planar-1976,
    AUTHOR = {Garey, M. R. and Johnson, D. S. and Tarjan, R. Endre},
     TITLE = {The planar {H}amiltonian circuit problem is {NP}-complete},
   JOURNAL = {SIAM J. Comput.},
  FJOURNAL = {SIAM Journal on Computing},
    VOLUME = {5},
      YEAR = {1976},
    NUMBER = {4},
     PAGES = {704--714},
      ISSN = {0097-5397},
   MRCLASS = {05C35 (02G05)},
  MRNUMBER = {444516},
MRREVIEWER = {Andreas\ Blass},
       DOI = {10.1137/0205049},
       URL = {https://doi.org/10.1137/0205049},
}

@article {Demoucron-graphes-1964,
    AUTHOR = {Demoucron, G. and Malgrange, Y. and Pertuiset, R.},
     TITLE = {Graphes Planaires: Reconnaissance et Construction de
Representations Planaires Topologiques},
   JOURNAL = {Rev. Franc. Rech. Oper.},
  FJOURNAL = {Rev. Franc. Rech. Oper.},
      YEAR = {1964},
    NUMBER = {8},
     PAGES = {33--34},
}

@incollection {Burton-computational-2013,
    AUTHOR = {Burton, Benjamin A.},
     TITLE = {Computational topology with {R}egina: algorithms, heuristics
              and implementations},
 BOOKTITLE = {Geometry and topology down under},
    SERIES = {Contemp. Math.},
    VOLUME = {597},
     PAGES = {195--224},
 PUBLISHER = {Amer. Math. Soc., Providence, RI},
      YEAR = {2013},
      ISBN = {978-0-8218-8480-5},
   MRCLASS = {57-04 (57N10 57Q15)},
  MRNUMBER = {3186674},
       DOI = {10.1090/conm/597/11877},
       URL = {https://doi.org/10.1090/conm/597/11877},
}

@article {Rubinstein-traversing-2019,
    AUTHOR = {Rubinstein, J. Hyam and Segerman, Henry and Tillmann, Stephan},
     TITLE = {Traversing three-manifold triangulations and spines},
   JOURNAL = {Enseign. Math.},
  FJOURNAL = {L'Enseignement Math\'{e}matique},
    VOLUME = {65},
      YEAR = {2019},
    NUMBER = {1-2},
     PAGES = {155--206},
      ISSN = {0013-8584,2309-4672},
   MRCLASS = {57Q15 (57K30 57Q25)},
  MRNUMBER = {4057358},
MRREVIEWER = {J.\ P. E. Hodgson},
       DOI = {10.4171/lem/65-1/2-5},
       URL = {https://doi.org/10.4171/lem/65-1/2-5},
}

@article{Newman-foundation-1926,
	Author = {Newman, M. H. A.},
	Fjournal = {Proceedings of the Royal Academy of Sciences at Amsterdam},
	Journal = {Proc. Royal Acad. Amsterdam},
	Pages = {610--641},
	Title = {On the foundations of combinatorial Analysis Situs},
	Volume = {29},
	Year = {1926}}

@article{Alexander-combinatorial-1930,
	Author = {Alexander, James W.},
	Doi = {10.2307/1968099},
	Fjournal = {Annals of Mathematics. Second Series},
	Issn = {0003-486X},
	Journal = {Ann. of Math. (2)},
	Mrclass = {DML},
	Mrnumber = {1502943},
	Number = {2},
	Pages = {292--320},
	Title = {The combinatorial theory of complexes},
	Url = {https://doi.org/10.2307/1968099},
	Volume = {31},
	Year = {1930}}

@article{Pachner-bistellare-1978,
	Author = {Pachner, Udo},
	Doi = {10.1007/BF01226024},
	Fjournal = {Archiv der Mathematik},
	Issn = {0003-889X},
	Journal = {Arch. Math. (Basel)},
	Mrclass = {57C05 (52A25)},
	Mrnumber = {0488070},
	Mrreviewer = {Gotz Brunner},
	Number = {1},
	Pages = {89--98},
	Title = {Bistellare \"{A}quivalenz kombinatorischer {M}annigfaltigkeiten},
	Url = {https://doi.org/10.1007/BF01226024},
	Volume = {30},
	Year = {1978}}

@article{Moise-affine-1952,
	Author = {Moise, Edwin E.},
	Doi = {10.2307/1969769},
	Fjournal = {Annals of Mathematics. Second Series},
	Issn = {0003-486X},
	Journal = {Ann. of Math. (2)},
	Mrclass = {56.0X},
	Mrnumber = {0048805},
	Mrreviewer = {S. S. Cairns},
	Pages = {96--114},
	Title = {Affine structures in {$3$}-manifolds. {V}. {T}he triangulation theorem and {H}auptvermutung},
	Url = {https://doi.org/10.2307/1969769},
	Volume = {56},
	Year = {1952}}

@book{matveev_book,
	Author = {Matveev, Sergei V.},
	Edition = {Second},
	Publisher = {Springer},
	Title = {Algorithmic Topology and Classification of 3-Manifolds},
	Year = {2007}}

\address{Stephan Tillmann\\School of Mathematics and Statistics F07, The University of Sydney, NSW 2006 Australia\\{stephan.tillmann@sydney.edu.au\\-----}}

\address{Anastasiia Tsvietkova\\Department of Mathematics and Computer Science, Rutgers University-Newark, Newark, NJ 07102, USA\\{a.tsviet@rutgers.edu}}

\Addresses
                                                      
\end{document}